\documentclass[11pt]{amsart}

\usepackage{amsfonts, amsmath, amscd}
\usepackage[psamsfonts]{amssymb}

\usepackage{amssymb}

\usepackage[usenames]{color}

\newtheorem{theorem}{Theorem}[section]
\newtheorem{lemma}[theorem]{Lemma}
\newtheorem{corollary}[theorem]{Corollary}
\newtheorem{question}[theorem]{Question}
\newtheorem{remark}[theorem]{Remark}
\newtheorem{proposition}[theorem]{Proposition}
\newtheorem{definition}[theorem]{Definition}

\newtheorem{fact}[theorem]{Fact}

\numberwithin{equation}{section}

\newcommand{\CC}{C_k}
\newcommand{\NN}{\mathbb{N}}

\newcommand{\e}{\varepsilon}

\newcommand{\Tt}{\mathcal{T}}

\newcommand{\ff}{\mathbb{F}}

\newcommand{\KK}{\mathcal{K}}

\newcommand{\Nn}{\mathcal{N}}

\newcommand{\IR}{\mathbb{R}}

\newcommand{\Ss}{\mathbb{S}}
\newcommand{\Pp}{\mathsf{P}}

\newcommand{\Ff}{\mathfrak{F}}

\renewcommand{\phi}{\varphi}

\newcommand{\U}{\mathcal U}

\newcommand{\supp}{\mathrm{supp}}

\input xy
\xyoption{all}

\title[On reflexive groups and function spaces with a Mackey group topology]{On reflexive groups and function spaces with a Mackey group topology}

\author{S.~Gabriyelyan}
\address{Department of Mathematics, Ben-Gurion University of the
Negev, Beer-Sheva, P.O. 653, Israel}
\email{saak@math.bgu.ac.il}

\begin{document}

\begin{abstract}
We prove that every reflexive abelian group $G$ such that its dual group $G^\wedge$ has the $qc$-Glicksberg property is a Mackey group. We show that a reflexive abelian group of finite exponent is a Mackey group.  We prove that, for a realcompact space $X$, the space $\CC(X)$ is barrelled if and only if it is a Mackey group.
\end{abstract}


\maketitle

\section{Introduction}

For an abelian topological group $(G,\tau)$ we denote by $\widehat{G}$ the group of all continuous characters of $(G,\tau)$.  If $\widehat{G}$ separates the points of $G$, the group $G$ is called {\em maximally almost periodic} (MAP for short). The class $\mathcal{LQC}$ of all locally quasi-convex groups is the most important subclass of the class $\mathcal{MAPA}$ of all MAP abelian groups  (all relevant definitions see Section \ref{sec:Mackey-prelim}).

Two topologies  $\tau$ and $\nu$ on an abelian group $G$  are said to be {\em compatible } if $\widehat{(G,\tau)}=\widehat{(G,\nu)}$.
Being motivated by the classical Mackey--Arens theorem the following notion was  introduced and studied in \cite{CMPT}: a locally quasi-convex abelian group $(G,\tau)$ is called a {\em Mackey group in  $\mathcal{LQC}$} or simply a {\em Mackey group} if for every compatible locally quasi-convex group topology $\nu$ on $G$ it follows that $\nu\leq \tau$. Every barrelled locally convex space  
is a Mackey group by \cite{CMPT}. Since every reflexive locally convex space  $E$ is barrelled by \cite[Proposition 11.4.2]{Jar}, we obtain that $E$ is a Mackey group. This result motivates the following question:
\begin{question} \label{q:Mackey-refl}
Which reflexive abelian topological groups are Mackey groups?
\end{question}
In Section \ref{sec:Mackey-prelim} we obtain a sufficient condition on a reflexive group to be a Mackey group, see Theorem \ref{t:weak-Mackey}. Using Theorem \ref{t:weak-Mackey} we obtain a complete answer to Question \ref{q:Mackey-refl} for reflexive groups of finite exponent.
\begin{theorem} \label{t:Mackey-finite-exp}
Any reflexive abelian group $(G,\tau)$ of finite exponent is a Mackey group.
\end{theorem}
Note that any metrizable precompact abelian group of finite exponent is a Mackey group, see \cite[Example 4.4]{BTAVM}. So there are non-reflexive Mackey groups of finite exponent. If $G$ is a metrizable {\em reflexive} group, then $G$ must be {\em complete} by \cite[Corollary 2]{Cha}. So $G$ is a Mackey group by \cite[Theorem 4.2]{CMPT}. On the other hand, there are reflexive {\em non}-complete groups $G$ of finite exponent, see \cite{GRNT}. Such groups $G$ are also Mackey by Theorem  \ref{t:Mackey-finite-exp}.

For a Tychonoff space $X$ let $\CC(X)$ be the space  of all continuous real-valued functions on $X$  endowed with the compact-open topology. 
The relations between locally convex properties of $\CC(X)$  and topological properties of $X$ are illustrated by the following classical results, see \cite[Theorem 11.7.5]{Jar}.
\begin{theorem}[Nachbin--Shirota] \label{t:Mackey-barrelled-NS}
For a Tychonoff space $X$ the space $\CC(X)$ is barrelled if and only if every functionally bounded subset of $X$ has compact closure.
\end{theorem}
This theorem motivates the following question posed in \cite{Gab-Cp}: For which Tychonoff spaces $X$ the space $\CC(X)$  is a Mackey group? In the next theorem we obtain a partial answer to this question.
\begin{theorem} \label{t:Mackey-Ck-realcom}
For a realcompact space $X$, the space $\CC(X)$ is barrelled if and only if it is a Mackey group.
\end{theorem}

It is well-known (see \cite[Theorem 13.6.1]{Jar}) that a Tychonoff space $X$ is realcompact if and only if the space $\CC(X)$ is bornological. This result and Theorem \ref{t:Mackey-Ck-realcom} imply
\begin{corollary}
A bornological space $\CC(X)$ is barrelled if and only if it is a Mackey group.
\end{corollary}

We prove Theorem \ref{t:Mackey-Ck-realcom} in Section \ref{sec:Mackey-Ck-realcom}.


\section{Proof of Theorem \ref{t:Mackey-finite-exp}}  \label{sec:Mackey-prelim}


Denote by $\mathbb{S}$ the unit circle group and set $\Ss_+ :=\{z\in  \Ss:\ {\rm Re}(z)\geq 0\}$.
Let $G$ be an abelian topological group.  If $\chi\in \widehat{G}$, it is considered as a homomorphism from $G$ into $\mathbb{S}$.
A subset $A$ of $G$ is called {\em quasi-convex} if for every $g\in G\setminus A$ there exists   $\chi\in \widehat{G}$ such that $\chi(x)\notin \Ss_+$ and $\chi(A)\subseteq \Ss_+$.
If $A\subseteq G$ and $B\subseteq \widehat{G}$ set
\[
A^\triangleright :=\{ \chi\in \widehat{G}: \chi(A)\subseteq \Ss_+\}, \quad B^\triangleleft:=\{ g\in G: \chi(g)\in\Ss_+ \; \forall \chi\in B\}.
\]
Then $A$ is quasi-convex if and only if $A^{\triangleright\triangleleft}=A$.
An abelian topological group $G$ is called {\em locally quasi-convex} if it admits a neighborhood base at the neutral element $0$ consisting of quasi-convex sets.
The dual group $\widehat{G}$ of $G$ endowed with the compact-open topology is denoted by $G^{\wedge}$. The homomorphism $\alpha_G : G\to G^{\wedge\wedge} $, $g\mapsto (\chi\mapsto \chi(g))$, is called {\em the canonical homomorphism}. If $\alpha_G$ is a topological isomorphism the group $G$ is called {\em  reflexive}.

If $G$ is a $MAP$ abelian group, we  denote by $\sigma(G,\widehat{G})$ the {\em weak topology} on $G$, i.e., the smallest group topology on $G$ for which the elements of $\widehat{G}$ are continuous. In the dual group $\widehat{G}$, we denote by $\sigma(\widehat{G},G)$ the topology of pointwise convergence.

We use the next fact, see Proposition 1.5 of \cite{Ban}.
\begin{fact} \label{f:Mackey-Ban}
Let $U$ be a neighborhood of zero of an abelian topological group $G$. Then  $U^\triangleright$ is a compact subset of $(\widehat{G},\sigma(\widehat{G},G))$.
\end{fact}

Let $G$ be a $MAP$ abelian group and $\mathcal{P}$  a topological property. Denote by $\Pp(G)$  the set of all  subspaces of $G$ with $\mathcal{P}$.
Following \cite{ReT3},  $G$ {\em respects} $\mathcal{P}$ if $\Pp(G)=\Pp\big(G,\sigma(G,\widehat{G})\big)$.
Below we define  weak versions of respected properties. For a $MAP$ abelian group $G$, we denote by $\Pp_{qc}(G)$ the set of all {\em quasi-convex} subspaces of $G$ with $\mathcal{P}$.
\begin{definition} \label{def:qc-Glicksberg} {\em
Let $(G,\tau)$ be a $MAP$ abelian group. We say that
\begin{enumerate}
\item[{\rm (i)}] $(G,\tau)$  {\em respects}  $\mathcal{P}_{qc}$ if $\Pp_{qc}(G)=\Pp_{qc}\big(G,\sigma(G,\widehat{G})\big)$;
\item[{\rm (ii)}]  $(G,\tau)^\wedge$ {\em weak${}^\ast$  respects}  $\mathcal{P}$ if $\Pp(G^\wedge)=\Pp\big(\widehat{G},\sigma(\widehat{G},G)\big)$;
\item[{\rm (iii)}] $(G,\tau)^\wedge$ {\em weak${}^\ast$  respects}  $\mathcal{P}_{qc}$ if $\Pp_{qc}(G^\wedge)=\Pp_{qc}\big(\widehat{G},\sigma(\widehat{G},G)\big)$.
\end{enumerate} }
\end{definition}

In the case $\mathcal{P}$ is the property $\mathcal{C}$ to be a compact subset and  a $MAP$ abelian group $(G,\tau)$  (or $G^\wedge$) (weak${}^\ast$) respects $\mathcal{P}$ or $\mathcal{P}_{qc}$, we shall say that  the group $G$ (or $G^\wedge$) has the ({\em weak${}^\ast$}) {\em Glicksberg property} or {\em $qc$-Glicksberg property}, respectively. So $G$ has the {\em Glicksberg property} or {\em respects compactness} if any weakly compact subset of $G$ is also compact in the original topology $\tau$. By a famous result of Glicksberg, any abelian locally compact group respects compactness. Clearly, if a $MAP$ abelian group $(G,\tau)$ has the Glicksberg property, then it also has the $qc$-Glicksberg property, and   if $(G,\tau)^\wedge$ has the  weak${}^\ast$ Glicksberg property, then it has also the  weak${}^\ast$  $qc$-Glicksberg property.
\begin{proposition} \label{p:Mackey-weak-qc-Glick}
Let $(G,\tau)$ be a locally quasi-convex group such that the canonical homomorphism $\alpha_G$ is continuous. If $(G,\tau)^\wedge$ has the weak${}^\ast$ $qc$-Glicksberg property, then  $(G,\tau)$ is a Mackey group.
\end{proposition}
\begin{proof}
Let $\nu$ be a locally quasi-convex topology on $G$ compatible with $\tau$ and let $U$ be a quasi-convex $\nu$-neighborhood of zero. Fact \ref{f:Mackey-Ban} implies that the quasi-convex subset $K:=U^\triangleright$ of $\widehat{G}$ is $\sigma(\widehat{G},G)$-compact, and hence $K$ is a compact subset of  $G^\wedge$ by the weak${}^\ast$ $qc$-Glicksberg property. Note that, by definition, $K^\triangleright$ is a neighborhood of zero in $G^{\wedge\wedge}$. As $\alpha_G$ is continuous,  $U=K^\triangleleft =\alpha_G^{-1}(K^\triangleright)$ is a neighborhood of zero in $G$. Hence $\nu\leq\tau$. Thus $(G,\tau)$ is a Mackey group.
\end{proof}

The following theorem gives a partial answer to Question \ref{q:Mackey-refl}.
\begin{theorem} \label{t:weak-Mackey}
Let $(G,\tau)$ be a reflexive abelian group. If $(G,\tau)^\wedge$ has the $qc$-Glicksberg property (in particular, the Glicksberg property), then  $(G,\tau)$ is a Mackey group.
\end{theorem}
\begin{proof}
Since $G$ is a reflexive group, the weak${}^\ast$ $qc$-Glicksberg property coincides with the $qc$-Glicksberg property, and  Proposition \ref{p:Mackey-weak-qc-Glick} applies.
\end{proof}

\begin{remark} \label{rem:Mackey-c0(S)} {\em
In Theorem \ref{t:weak-Mackey} the reflexivity of $G$ is essential. Indeed, let $G$ be a proper dense subgroup of a compact metrizable abelian group $X$. Then $G^\wedge =X^\wedge$ (see \cite{Aus,Cha}), and hence the discrete group $G^\wedge$ has the Glicksberg property. Now set $c_0 (\Ss):= \{ (z_n) \in \Ss^\mathbb{N} :\; z_n \to 1\}$. Denote by $\mathfrak{p}_0$ the product topology on $c_0 (\Ss)$ induced from $\Ss^\NN$,  and let $\mathfrak{u}_0$ be the uniform topology on $c_0 (\Ss)$ induced by the metric $d\big((z_n^1), (z_n^2 ) \big)= \sup \{ |z_n^1 -z_n^2 |, n\in\NN \}$. Then, by \cite[Theorem 1]{Gab}, $\mathfrak{p}_0$ and $\mathfrak{u}_0$ are locally quasi-convex and compatible topologies on $c_0 (\Ss)$ such that $\mathfrak{p}_0 <\mathfrak{u}_0$. Thus the group $G:=\big( c_0 (\Ss),\mathfrak{p}_0)$ is a precompact arc-connected metrizable group such that $G^\wedge$ has the Glicksberg property, but $G$ is {\em not} a Mackey group.}
\end{remark}

Theorem \ref{t:weak-Mackey} motivates the following question: {\em For which (reflexive) abelian groups $G$ the dual group $G^\wedge$ has the (weak${}^\ast$, weak${}^\ast$ $qc$-, $qc$-) Glicksberg property}? Below in Propositions \ref{p:Mackey-g-barrelled} and \ref{p:nuclear-Mackey} we give some sufficient conditions on $G$ for which $G^\wedge$ has the Glicksberg property.

Recall (see \cite{CMPT}) that a $MAP$ abelian group $G$ is called {\em $g$-barrelled} if any $\sigma(\widehat{G},G)$-compact subset of  $\widehat{G}$ is equicontinuous. Every barrelled locally convex space $E$ is a $g$-barrelled group, but the converse does not hold in general, see \cite{CMPT}. Every locally quasi-convex $g$-barrelled abelian group $G$ is a Mackey group by Theorem 4.2 of \cite{CMPT}.
\begin{proposition} \label{p:Mackey-g-barrelled}
If $G$ is a $g$-barrelled group, then $G^\wedge$ has the Glicksberg property.
\end{proposition}
\begin{proof}
Let $K$ be a $\sigma(\widehat{G},G^{\wedge\wedge})$-compact subset of  $\widehat{G}$. Then $K$ is $\sigma(\widehat{G},G)$-compact as well, so $K$ is equicontinuous. Hence there is a neighborhood $U$ of zero in $G$ such that $K\subseteq U^\triangleright$, see Corollary 1.3 of \cite{CMPT}. The set $U^\triangleright$ is a compact subset of $G^\wedge$ by Fact \ref{f:Mackey-Ban}. As $K$ is also a closed subset of $G^\wedge$, we obtain that $K$ is compact in $G^\wedge$. Thus $G^\wedge$ has the Glicksberg property.
\end{proof}
For reflexive groups this proposition can be reversed.
\begin{proposition} \label{p:Mackey-g-bar-refl}
If $G$ is a reflexive group, then $G$ is $g$-barrelled if and only if $G^\wedge$ has the Glicksberg property.
\end{proposition}
\begin{proof}
Assume that $G^\wedge$ has the Glicksberg property and $K$ is  a $\sigma(\widehat{G},G)$-compact subset of  $\widehat{G}$. By the reflexivity of $G$, $K$ is also $\sigma(\widehat{G},G^{\wedge\wedge})$-compact. So $K$ is compact in $G^\wedge$ by the Glicksberg property. Therefore $K^\triangleright$ is a neighborhood of zero in $G^{\wedge\wedge}$. So, by the reflexivity of $G$, $K^\triangleleft =\alpha_G^{-1}(K^\triangleright)$ is a neighborhood of zero in $G$. Since $K\subseteq K^{\triangleleft\triangleright} $ we obtain that $K$ is equicontinuous, see Corollary 1.3 of \cite{CMPT}. Thus $G$ is $g$-barrelled. The converse assertion follows from Proposition \ref{p:Mackey-g-barrelled}.
\end{proof}

Recall that a topological group $X$ is said to have a {\it subgroup topology} if it has a base at the identity consisting of subgroups. For the definition and properties of nuclear groups, see \cite{Ban}.
\begin{lemma} \label{l:Mackey-linear}
Let $G$ be an abelian topological group with a subgroup topology. Then $G$ is a  locally quasi-convex nuclear group and has  the  Glicksberg property.
\end{lemma}
\begin{proof}
By Proposition 2.2 of \cite{AG}, $G$ embeds into a product of discrete groups. Therefore $G$ is a  locally quasi-convex nuclear group by Propositions 7.5 and 7.6 and Theorem 8.5 of \cite{Ban}. Finally, the group $G$ has the Glicksberg property by  \cite{BaMP2}.
\end{proof}

To prove Theorem \ref{t:Mackey-finite-exp} we need the following proposition.
\begin{proposition} \label{p:nuclear-Mackey}
Let $(G,\tau)$ be a locally quasi-convex abelian group of finite exponent. Then $(G,\tau)$ and hence also $(G,\tau)^\wedge$ have the  Glicksberg property.
\end{proposition}
\begin{proof}
Propositions 2.1  of \cite{AG} implies that the topologies of the groups $(G,\tau)$ and  $(G,\tau)^\wedge$ are subgroup topologies, and Lemma \ref{l:Mackey-linear} applies.
\end{proof}

\begin{proof}[Proof of Theorem \ref{t:Mackey-finite-exp}]
Since $(G,\tau)$ is locally quasi-convex, Proposition \ref{p:nuclear-Mackey} implies that  $(G,\tau)^\wedge$ has the Glicksberg property. Thus  $(G,\tau)$ is a Mackey group by Theorem \ref{t:weak-Mackey}.
\end{proof}

For  Tychonoff spaces $X$ and $Y$ we denote by $\CC(X,Y)$ the space  of all continuous functions from $X$ into $Y$ endowed with the compact-open topology. 
R.~Pol and F.~Smentek \cite{PolSmen} proved that the group $\CC(X,D)$ is reflexive for every  finitely generated discrete group $D$ and each zero-dimensional realcompact $k$-space $X$. This result and Theorem \ref{t:Mackey-finite-exp}  immediately imply
\begin{corollary} \label{c:Mackey-C(X,F)-discrete}
Let $X$ be a zero-dimensional realcompact $k$-space and $\ff$ be a finite abelian group. Then  $\CC(X,\ff)$ is  a Mackey group.
\end{corollary}



We end this section with the following two questions. We do not know whether the converse in Theorem  \ref{t:weak-Mackey} is true.
\begin{question} \label{question-weak-Glick}
Let $G$ be a reflexive Mackey group. Does $G^\wedge$ have the $qc$-Glicksberg property?
\end{question}

 Set $\mathfrak{F}_0(\Ss):=(c_0(\Ss),\mathfrak{u}_0)$, see Remark \ref{rem:Mackey-c0(S)}. Then the group $\mathfrak{F}_0(\Ss)$ is reflexive \cite{Gab} and does not have the Glicksberg property by \cite{Ga-Top}. These results motivate the following question.
\begin{question} \label{question-free}
Does $\mathfrak{F}_0(\Ss)$ have the $qc$-Glicksberg property?
\end{question}
This question is of importance because the dual group $\mathfrak{F}_0(\Ss)^\wedge$ is the free abelian topological group $A(\mathfrak{s})$ over a convergent sequence $\mathfrak{s}$, see \cite{Gab}. So the positive answer to this question together with  Theorem \ref{t:weak-Mackey} would imply: (1)  the group $A(\mathfrak{s})$ is a Mackey group, answering in the affirmative a question posed in \cite{Gab-Mackey}, and (2) there are locally quasi-convex (even reflexive and Polish) abelian groups with the $qc$-Glicksberg property but without the Glicksberg property. On the other hand, under the assumptions that Question \ref{question-free} has a negative answer and the group $A(\mathfrak{s})$ is Mackey, we obtain a negative answer to Question \ref{question-weak-Glick}.


\section{Proof of Theorem \ref{t:Mackey-Ck-realcom}} \label{sec:Mackey-Ck-realcom}


Let $E$ be a nontrivial locally convex space and denote by $E'$ the topological dual space of $E$. Clearly,  $E$ is also an abelian topological group. Therefore we can consider the group $\widehat{E}$ of all continuous characters of $E$. The next important result is proved in \cite{HeZu,Smith}, see also \cite[23.32]{HR1}.
\begin{fact} \label{f:Mackey-group-lcs}
Let $E$ be a locally convex space. Then the mapping $p:E' \to \widehat{E}$, defined by the equality
\[
p(f)=\exp\{ 2\pi i f\}, \quad \mbox{ for all } f\in E',
\]
is a group isomorphism between $E'$ and $\widehat{E}$.
\end{fact}

Recall that the dual space of $\CC(X)$ is the space $M_c(X)$ of all Borel measures $\mu$ on $X$ with compact support $\supp(\mu)$, see \cite[Corollary 7.6.5]{Jar}. For a point $x\in X$ we denote by $\delta_x$ the point measure associated with the linear form $f\mapsto f(x)$.

The next lemma is crucial for our proof of Theorem \ref{t:Mackey-Ck-realcom}.
\begin{lemma} \label{l:Mackey-Ck}
Let $A$ be a functionally bounded subset of a Tychonoff space $X$. If there is a discrete family $\U=\{ U_n\}_{n\in\NN}$  of open subsets of $X$ such that $U_n \cap A \not=\emptyset$ for every $n\in\NN$, then $\CC(X)$ is not a Mackey group.
\end{lemma}

\begin{proof}
For every $n\in\NN$, take arbitrarily $x_n\in U_n\cap A$ and  set $\chi_n := (1/n) \delta_{x_n}$. Since $A$ is functionally bounded, we obtain that $\chi_n \to 0$ in the weak* topology on $M_c(X)$. Denote by $Q:c_0\to \Ff_0 (\Ss)$ the quotient map, so $Q\big( (x_n)_{n\in\NN}\big) = \big( (q(x_n))_{n\in\NN}\big)$, where $q:\IR\to \Ss$ is defined by $q(x)=\exp\{ 2\pi i x\}$. Now we can define the linear injective operator $F: C(X) \to \CC(X)\times c_0$ and the monomorphism $F_0 : C(X) \to \CC(X)\times \Ff_0(\Ss)$  setting ($\forall f\in C(X)$)
\[
\begin{split}
F(f) & := \big( f, R(f)\big), \mbox{ where } R(f):= \big(\chi_n(f)\big) \in c_0, \\
F_0(f)& := \big( f, R_0(f)\big), \mbox{ where } R_0(f):= Q\circ R(f)=\big( \exp\{ 2\pi i \chi_n(f)\} \big) \in \Ff_0(\Ss).
\end{split}
\]
Denote by $\Tt$ and $\Tt_0$ the topology on $C(X)$ induced from $\CC(X)\times c_0$ and $\CC(X)\times \Ff_0(\Ss)$, respectively. So $\Tt$ is a locally convex vector topology on $C(X)$ and $\Tt_0$ is a locally quasi-convex group topogy on $C(X)$ (since the group $\Ff_0(\Ss)$ is locally quasi-convex, and a subgroup of a product of locally quasi-convex groups is clearly locally quasi-convex). Denote by $\tau_k$ the compact-open topology on $C(X)$. Then, by construction, $\tau_k \leq \Tt_0 \leq\Tt$, so taking into account Fact \ref{f:Mackey-group-lcs} we obtain
\begin{equation}\label{equ:Mackey-Ck-1}
p\big( M_c(X) \big) \subseteq \widehat{\big( C(X),\Tt_0\big)} \subseteq p\big(  (C(X),\Tt)' \big).
\end{equation}

Let us show that the topologies $\tau_k$ and $\Tt_0$ are compatible. By (\ref{equ:Mackey-Ck-1}), it is enough to show that each character of $\big( C(X),\Tt_0\big)$ belongs to $p\big( M_c(X)\big)$. Fix $\chi\in \widehat{\big( C(X),\Tt_0\big)}$. Then  (\ref{equ:Mackey-Ck-1}) and the Hahn--Banach extension theorem imply that $\chi=p(\eta)$ for some
\[
\eta =\big(\nu,(c_n)\big)\in M_c(X) \times \ell_1, \mbox{ where } \nu \in M_c(X) \mbox{ and } (c_n)\in \ell_1,
\]
and
\[
\eta (f)=\nu(f) + \sum_{n\in\NN} c_n \chi_n (f), \quad \forall f\in C(X).
\]
To prove that $\chi\in p\big( M_c(X)\big)$ it is enough to show that $c_n=0$ for almost all indices $n$. Suppose for a contradiction that  $|c_n|>0$ for infinitely many indices $n$. Take a neighborhood $U$ of zero in $\Tt_0$ such that (see Fact \ref{f:Mackey-group-lcs})
\begin{equation} \label{equ:Mackey-Ck-9}
\eta(U) \subseteq \left( -\frac{1}{10},\frac{1}{10}\right) + \mathbb{Z}.
\end{equation}
We can assume that $U$ has a canonical form
\[
U=F_0^{-1}\left( \left( \{ f\in C(X): f(K)\subset (-\e,\e)\} \times \big(V^\NN \cap c_0(\Ss)\big)\right) \cap F_0\big(C(X)\big) \right),
\]
for some compact set $K\subseteq  X$, $\e >0$ and a neighborhood $V$ of the identity of $\Ss$.
For every $n\in\NN$, choose a continuous function $g_n: X\to [0,1]$ such that  $g_n(x_n)=1$  and $g_n(X\setminus U_n)=\{ 0\}$. So, by  the discreteness of $\U$, we obtain
\begin{equation} \label{equ:Mackey-Ck-3}
\chi_n(g_n)=\frac{1}{n}, \quad  \mbox{ and  } \;  \chi_m(g_n)=0 \; \mbox{  for every distinct $n,m\in\NN$  }.
\end{equation}

Let $C=\supp(\nu)$, so $C$ is a compact subset of $X$. Then the discreteness of the family $\U$ implies that there is $n_0\in \NN$ such that $U_n \cap (K\cup C)=\emptyset$ for every $n>n_0$. Since $|c_n|>0$ for infinitely many indices, we can find an index $\alpha >n_0$ such that $0<|c_\alpha|<1/100$ (recall that $(c_n)\in\ell_1$). Set
\[
h(x) = \left[ \frac{1}{4c_\alpha} \right] \cdot \alpha \cdot g_\alpha (x),
\]
were $[x]$ is the integral part of a real number $x$. 
Then  (\ref{equ:Mackey-Ck-3}) implies that 
\begin{equation} \label{equ:Mackey-Ck-4}
\chi_n(h)=0 \mbox{ if } n\not= \alpha, \; \mbox{ and } \; \chi_\alpha (h) = \left[ \frac{1}{4c_\alpha} \right] \cdot \alpha \cdot \frac{1}{\alpha} =\left[ \frac{1}{4c_\alpha} \right] \in \mathbb{Z}.
\end{equation}
Therefore $R_0(h)$ is the identity of $\mathfrak{F}_0(\Ss_0)$. Since also $h\in \{ f\in C(X): f(K)\subset (-\e,\e)\}$ we obtain that $h\in U$. Noting that $\nu(h)=0$ and setting $r_\alpha := \frac{1}{4c_\alpha} - \left[ \frac{1}{4c_\alpha}\right]$ (and hence $0\leq r_\alpha <1$), (\ref{equ:Mackey-Ck-4}) implies
\[
\frac{1}{4}-\frac{1}{100} < \eta (h)= c_\alpha \chi_\alpha (h)=  c_\alpha  \left[ \frac{1}{4c_\alpha} \right]=  c_\alpha  \left( \frac{1}{4c_\alpha} - r_\alpha\right) = \frac{1}{4} -c_\alpha r_\alpha <\frac{1}{4}+\frac{1}{100}.
\]
But these inequalities contradict the inclusion (\ref{equ:Mackey-Ck-9}). This contradiction shows that $c_n=0$ for almost all indices $n$, and hence $\eta\in M_c(X)$. Thus $\tau_k$ and $\Tt_0$ are compatible.

To complete the proof we have to show that the topology  $\Tt_0$ is strictly finer than $\tau_k$. First we note that $(n/2)g_n \to 0$ in $\tau_k$. Indeed, let $K_0$ be a compact subset  of $X$ and $\e >0$. Since the family $\U$ is discrete, there is $N\in\NN$ such that $U_n\cap K_0 =\emptyset$ for every $n> N$. Then $(n/2)g_n \in \{ f\in C(X): f(K_0)\subset (-\e,\e)\}$ for $n>N$. On the other hand, 
\[
F_0\big((n/2)g_n\big) =\big( (n/2)g_n, (\exp\{ 2\pi i \chi_k\big( (n/2)g_n\big)\} ) \big) =  \big( (n/2)g_n, (0,\dots,0,-1,0,\dots)\big),
\]
where $-1$ is placed in position $n$. So $(n/2)g_n\not\to 0$ in $\Tt_0$. Thus $\Tt_0$ is strictly finer than $\tau_k$.
\end{proof}

Now we are ready to prove Theorem \ref{t:Mackey-Ck-realcom}.
\begin{proof}[Proof of Theorem \ref{t:Mackey-Ck-realcom}]
Assume that $\CC(X)$ is a Mackey group. Let us show that  every functionally bounded subset of $X$ has compact closure.
Suppose for a contradiction that there is a closed functionally bounded subset $A$ of $X$ which  is not compact. Since $A$ is a closed subset of a realcompact space, there is a continuous real-valued function $f$ on  $X$ such that $f|_A$ is unbounded, see \cite[Problem 8E.1]{GiJ}. So there exists a discrete sequence of open sets $\U=\{ U_n\}_{n\in\NN}$ intersecting $A$. Therefore $\CC(X)$ is not a Mackey group by Lemma \ref{l:Mackey-Ck}. This contradiction shows that every functionally bounded subset of $X$ has compact closure. By the Nachbin--Shirota theorem \ref{t:Mackey-barrelled-NS}, the space $\CC(X)$ is barrelled.

Conversely, if $\CC(X)$ is a barrelled locally convex space, then it is a Mackey group by Proposition 5.4 of \cite{CMPT}.
\end{proof}
We do not know whether the assumption to be a realcompact space can be omitted in Theorem \ref{t:Mackey-Ck-realcom}.
\begin{question}
Let $X$ be a Tychonoff space. Is it true that $\CC(X)$ is barrelled if and only if it a Mackey group?
\end{question}

\bibliographystyle{amsplain}

\end{document}